\documentclass[12pt]{amsart}
\usepackage{amssymb, txfonts}

\usepackage{helvet} % for Sans Serif fonts

\usepackage{courier} % default ttdefault

\usepackage{eucal} % better calligraphic fonts

\reversemarginpar %marginal notes appear on the left side margin
\normalmarginpar %switches back to right side margin

\let\oldmarginpar\marginpar %changes the font of margin text
\renewcommand\marginpar[1]{\-\oldmarginpar{\raggedright\small\sf #1}}

\title[Coordinate  rings of affine group schemes]{On the finite generation of coordinate rings of affine group schemes over discrete valuation rings}

\author{Najmuddin Fakhruddin}
 
\address{School of Mathematics, Tata Institue of Fundamental Research, 
Homi Bhabha Road, Mumbai 400005, India}

\email{naf@math.tifr.res.in}

\newcommand{\nc}{\newcommand}

\nc{\rnc}{\renewcommand}

\nc{\bs}{\backslash}
\nc{\te}{\otimes}
\nc{\lf}{\lfloor} %for round down
\nc{\rf}{\rfloor}
\nc{\lc}{\lceil}  %for round up
\nc{\rc}{\rceil}
\nc{\lr}{\longrightarrow}
\nc{\sr}{\stackrel}
\nc{\dar}{\dashrightarrow}
\nc{\thra}{\twoheadrightarrow}

\nc{\mc}{\mathcal}
\nc{\mb}{\mathbb}
\nc{\mf}{\mathbf}
\nc{\mr}{\mathrm}

\nc{\bP}{\mathbb{P}}
\rnc{\P}{\mathbb{P}}
\nc{\Q}{\mathbb{Q}}
\nc{\Z}{\mathbb{Z}}
\nc{\C}{\mathbb{C}}
\nc{\R}{\mathbb{R}}
\nc{\A}{\mathbb{A}}
\nc{\V}{\mathbb{V}}
\nc{\W}{\mathbb{W}}
\nc{\N}{\mathbb{N}}

\nc{\aff}{{\A}^1}
\nc{\naive}{\!\sim_n}
\nc{\Spec}{\mr{Spec}}

\nc{\omx}{\omega_X}

\nc{\ep}{\epsilon}
\nc{\ve}{\varepsilon}

\nc{\wt}{\widetilde}
\nc{\wh}{\widehat}
\nc{\ol}{\overline}
\rnc{\sl}{\shoveleft}

\newtheorem{thm}{Theorem}

\newtheorem{cor}[thm]{Corollary}
\newtheorem{lem}[thm]{Lemma}

\theoremstyle{definition}

\newtheorem{rems}[thm]{Remarks}

\input{xypic}

\begin{document}

\begin{abstract}
  We prove a finite generation result for the coordinate ring of
  certain affine group schemes over a discrete valuation ring. This
  may be used to simplify the use of results of Prasad and Yu on
  quasi-reductive groups by Mirkovic and Vilonen in their work on
  geometric Langlands duality.
\end{abstract}

\maketitle

In \cite{prasad-yu}, Prasad and Yu defined quasi-reductive group
schemes and proved some structure theorems for them. Their work was 
motivated by a question of Vilonen, and their results are required in the
work of Mirkovic and Vilonen \cite{mirkovic-vilonen} on the geometric
construction of the Langlands dual group over arbitrary fields. One of
their results is that quasi-reductive group schemes are always of
finite type. In this note we give a simple proof of a more general
finite generation result which is essentially the best possible.

\section{}

Let $R$ be a DVR with quotient field $K$, residue field $k$, and let
$\ol{k}$ be an algebraic closure of $k$.

\begin{thm}

  Let $\mc{G}$ be a flat affine group scheme over $R$ and assume that
  $R$ is excellent.  If the generic fibre $\mc{G}_K$ is reduced and of
  finite type over $K$, the reduced special fibre
  $(\mc{G}_k)_{\mathrm{red}}$ is of finite type over $k$ and
  $\dim(\mc{G}_K) = \dim(\mc{G}_k)$, then $\mc{G}$ is of finite type
  over $R$.
  \label{thm:1}
\end{thm}

The following is an immediate consequence.
\begin{cor}
  Let $\mc{G}$ be a flat affine group scheme over (an arbitrary DVR)
  $R$.  If the generic fibre $\mc{G}_K$ is smooth and of finite type
  over $K$, the reduced geometric special fibre
  $(\mc{G}_{\ol{k}})_{\mathrm{red}}$ is of finite type over $\ol{k}$
  and $\dim(\mc{G}_K) = \dim(\mc{G}_{\ol{k}})$, then $\mc{G}$ is of
  finite type over $R$.
  \label{cor:1}
\end{cor}

\begin{proof}
By faithfully flat descent we may replace $R$ by a complete discrete
valuation ring with algebraically closed residue field. Such
a ring is excellent, hence we may apply Theorem \ref{thm:1} to
conclude.
\end{proof}

\begin{proof}[Proof of Theorem \ref{thm:1}]
  Since $R$ is excellent, the generic fibre of $\mc{G} \times_R
  \hat{R}$ is reduced, hence by faithfully flat descent we may assume
  that $R$ is complete.  Let $\mc{A} = \Gamma(\mc{G},
  \mc{O}_{\mc{G}})$ be the coordinate ring of $\mc{G}$. By the
  discussion in \S 5.3 of \cite{prasad-yu} (here we use that $R$ is
  complete), we may write $\mc{G} = \varprojlim_{i \in I} \mc{G}_i$,
  where $I$ is a directed set and $\mc{G}_i$ are flat finite type
  affine group schemes over $R$.  Since $\mc{G}_K$ is of finite type
  over $K$ we may assume that the induced morphisms $\mc{G}_K \to
  (\mc{G}_i)_K$ are isomorphisms for all $i \in I$, so that we have
  $\mc{A} = \varinjlim_{i \in I} \mc{A}_i$, where $\mc{A}_i =
  \Gamma(\mc{G}_i, \mc{O}_{\mc{G}_i})$ and all the induced maps
  $\mc{A}_i \to \mc{A}$ are injective.  Since tensor products commute
  with direct limits, we have $\varinjlim_{i \in I} \mc{A}_i \otimes_R
  k = \mc{A} \otimes_R k$. Since the $\mc{G}_i$ are flat and of finite
  type, we have $\dim((\mc{G}_i)_k) = \dim(\mc{G}_k)$ for all $i \in
  I$.

  The group scheme $(\mc{G}_k)_{\mathrm{red}}$ is of finite type, so
  it has only finitely many connected (= irreducible) components. For
  each $i$, let $\mc{G}_{i,k}'$ be the union of those components of
  $(\mc{G}_i)_k$ which contain points of $\pi_i(\mc{G}_k)$, where
  $\pi_i: \mc{G} \to \mc{G}_i$ are the maps in the directed system.
  Since the $\pi_i$ are homomorphisms of group schemes, it follows
  that $\mc{G}_{i,k}'$ is an affine open subgroup scheme of
  $(\mc{G}_i)_k$. By \cite[2.2.6]{bruhat-tits2}, there exist affine
  open subschemes $\mc{G}_i'$ of $\mc{G}_i$ such that $(\mc{G}_i')_k =
  \mc{G}_{i,k}'$ for all $i \in I$. The $\mc{G}_i'$ are subgroup
  schemes of $\mc{G}_i$ and by construction the maps $\pi_i$ factor
  through the inclusion $\mc{G}_i' \subset \mc{G}_i$.  If $\mc{A}_i'$
  denotes the coordinate ring of $\mc{G}_i'$, then the induced map
  $\varinjlim_{i \in I}\mc{A}_i' \to \mc{A}$ is an isomorphism, hence
  $\mc{G} = \varprojlim_{i \in I}\mc{G}_i'$.  So by replacing
  $\mc{G}_i$ with $\mc{G}_i'$, we may assume that the map on component
  groups induced by the morphisms $\mc{G}_k \to (\mc{G}_i)_k$ are all
  isomorphisms.  By Lemma \ref{lem:bij} we may now also assume that
  all the morphisms $(\mc{G}_k)_{\mr{red}} \to
  ((\mc{G}_i)_k)_{\mr{red}}$ are isomorphisms.

  Let $\wt{\mc{G}_i}$ (resp. $\wt{\mc{G}}$) be the normalisation of
  $\mc{G}_i$ (resp. $\mc{G}$), so that we have a commutative diagram:
  \[
  \xymatrix{ \wt{\mc{G}} = \varprojlim \wt{\mc{G}_i} \ar[r]
    \ar[d]_{\nu} &
    \wt{\mc{G}_i} \ar[r]^{\wt{\pi_{i,j}}} \ar[d]_{\nu_i} & \wt{\mc{G}_j} \ar[d]_{\nu_j}\\
    \mc{G} \ar[r] & \mc{G}_i \ar[r]^{\pi_{i,j}} & \mc{G}_j \\
  }
  \]
  Since $R$ is complete, hence excellent, the $\nu_i$ are finite
  morphisms. The $\pi_{i,j}$ are finite type bijections, hence
  $\wt{\pi_{i,j}}$ are birational, quasi-finite and finite type
  morphisms. Since $\wt{\mc{G}_j}$ are also normal, it follows from
  Zariski's main theorem that the $\wt{\pi_{i,j}}$ are open
  immersions.  The morphisms induced by $\pi_{i,j}$ from
  $(\mc{G}_i)_k$ to $(\mc{G}_j)_k$ are finite since the induced
  morphisms on reduced schemes are isomorphisms.  Since the $\nu_i$
  are also finite, it follows that the morphisms from
  $(\wt{\mc{G}_i})_{k}$ to $(\wt{\mc{G}_j})_{k}$ induced by
  $\wt{\pi_{i,j}}$ are also finite.  The images of these morphisms are
  therefore both open and closed, hence a union of connected
  components.  It then follows that all the $\wt{\pi_{i,j}}$ are
  isomorphisms for $i,j$ in a cofinal subset of $I$.  By replacing $I$
  with such a subset, we may assume that $\wt{\mc{G}} = \wt{\mc{G}_i}$
  for all $i$; in particular, we may assume that $\wt{\mc{G}}$ is of
  finite type over $R$.

  Now for any $i$, we have morphisms
  \[
  \wt{\mc{G}} \to \mc{G} \to \mc{G}_i
  \]
  and the morphism $\wt{\mc{G}} \to \mc{G}_i$ is finite. Since
  $\mc{G}_i$ is a noetherian scheme, it follows that $\mc{G}$ is
  finite over $\mc{G}_i$, hence of finite type over $R$.
\end{proof}

\begin{lem}
  Let $X = \varprojlim_{i \in I} X_i$ be a directed inverse limit of
  irreducible affine schemes over a field $k$, with $\dim(X) =
  \dim(X_i)$ for all $i$ and $X_{\mathrm{red}}$, $X_i$ of finite type
  over $k$ for all $i$. Then there exists $J \subset I$, a cofinal
  subset, such that the induced morphisms $X_{\mr{red}} \to
  (X_j)_{\mr{red}}$ are isomorphisms for all $j \in J$.
  \label{lem:bij}
\end{lem}

\begin{proof}
  Since $\varprojlim_{i \in I} (X_i)_{\mr{red}} = (\varprojlim_{i \in
    I} X_i)_{\mr{red}}$ we may assume that all the $X_i$ (hence also
  $X$) are reduced and $X$ is of finite type.

  Let $\pi_i:X \to X_i$ be the natural morphisms, $A := \Gamma(X,
  \mc{O}_X)$, $A_i := \Gamma(X_i, \mc{O}_{X_i})$ and $\sigma_i:A_i \to
  A$ the induced $k$-algebra homomorphsims.  Let $f_1,f_2,\dots,f_r$
  be generators of $A$. We may find a cofinal $J \subset I$ such that
  $f_k \in \sigma_j(A_j)$ for all $k = 1,2,\dots,r$ and $j \in J$.
  This implies that $\pi_j$ are closed embeddings for all $j \in J$.
  Since the $X_j$'s are irreducible, the dimension condition implies
  that $\pi_j$ are isomorphisms for all $j \in J$.
\end{proof}

\begin{rems} (1) We do not know whether Theorem \ref{thm:1} is true
  without assuming $\mc{G}$ affine (the assumption is used to write
  $\mc{G}$ as an inverse limit of finite type group schemes). \\
  (2) If $K$ is of characteristic $0$ then by Cartier's theorem the
  assumption that the generic fibre be reduced in Theorem \ref{thm:1}
  is redundant. However, if $\mathrm{char}(K) = p >0$ this is not the
  case as shown by the following example: Let $\mc{G} = \varprojlim_{n
    \in \mb{N}}\alpha_{p/R}$ where the transition maps are all equal
  to the endomorphism of $\alpha_{p/R}$ induced by multiplication by a
  uniformizer of $R$. Then $\mc{G}_K \cong \alpha_{p/K}$, $\mc{G}_{k}
  \cong \mathrm{Spec}(k)$ and $\mc{G}$ is
  flat over $R$.  However $\mc{G}$ is not of finite type over $R$. \\
  (3) By faithfully flat descent, Theorem \ref{thm:1} and Corollary
  \ref{cor:1} have obvious analogues for group schemes over Dedekind
  schemes such that the group scheme is of finite type over a Zariski
  open subset of the base. The group scheme $\mr{Spec}
  (\Z[x/2,x/3,x/5,\dots])$ over $\Z$ (\cite[Remark 4.7]{onoda}), all
  of whose fibres are isomorphic to $\mb{G}_a$ over the corresponding
  residue field, shows that in general the finite type condition over
  a Zariski open subset of the base cannot be replaced by a more local
  condition. However, Prasad and Yu have shown \cite[Theorem
  1.5]{prasad-yu} that this can be done when all the fibres are
  reductive \\
  (4) Onoda \cite{onoda} has proved finite generation results for a
  large class of affine schemes over more general bases. However, his
  results do not apply in the situation of Theorem \ref{thm:1} since
  his assumption on the special fibre is not satisfied.
\end{rems}

{\bf Acknowledgements.} I thank Brian Conrad and Gopal Prasad for
their comments; in particular, the proof of Lemma \ref{lem:bij}, which
is simpler than the original, was suggested by Brian Conrad. I also
thank S.~M.~Bhatwadekar for informing me of the article \cite{onoda}
and for related discussions.

The first version of this note was written while I was visiting the
Korea Institute of Advanced Study; I thank Andreas Bender and the
staff of KIAS for their hospitality.

\def\cprime{$'$}

\end{document}